\documentclass{article}

\usepackage{amsmath,amsfonts}%
\usepackage{amsthm}%
\usepackage{amsmath,amssymb}
\usepackage{array}
\usepackage{bm}
\usepackage[english]{babel}
\usepackage{comment}
\usepackage{hyperref}

\theoremstyle{plain}
 \newtheorem{thm}{Theorem}[section]
 \newtheorem{cor}[thm]{Corollary}
\newtheorem{lem}[thm]{Lemma}
\newtheorem{prop}[thm]{Proposition}
 \theoremstyle{definition}
 \newtheorem{defn}[thm]{Definition}
 \theoremstyle{remark}
 \newtheorem{rem}[thm]{Remark}
 
 \numberwithin{equation}{section}

\begin{document}
\title{Composition operators on weighted\\ modulation spaces}
\author{Héctor Ariza, Carmen Fernández, Antonio Galbis}
\date{}


\maketitle

\begin{abstract}
We study composition operators whose symbols are suitable perturbations of the identity and which act between different weighted modulation classes. We consider both modulation spaces formed by tempered distributions and those whose elements are ultradistributions defined in terms of a subadditive weight.\\
\end{abstract}

\textbf{Mathematics Subject Classification (2010):} Primary 47G30; Secondary 42B35, 46F05, 47B33. \\

   \textbf{Keywords:} Composition operators, pseudodifferential operators,  modulation spaces, Gelfand-Shilov spaces. \\

\textbf{Emails:} hecare@alumni.uv.es, fernand@uv.es, antonio.galbis@uv.es\\

   The research of C. Fern\'andez and A. Galbis is  supported by  project PID2024-162128NB-I00, Agencia Estatal de Investigaci\'on, Ministerio de Ciencia, Investigaci\'on y Universidades.\\

\par\medskip\noindent
\newpage

Composition operators on various spaces of smooth functions have been extensively studied during the past ten years. A central problem is to find necessary and/or sufficient conditions under which the operator on a given space is well defined and continuous. For the Schwartz space ${\mathcal S}({\mathbb R})$ this was settled in \cite{GJ_18}. In a recent paper \cite{AJK_25} the authors consider weighted composition operators in ${\mathcal S}({\mathbb R}^d). $  The study of composition operators in Gelfand-Shilov  spaces ${\mathcal S}_\omega ({\mathbb R})$ began in \cite{jmaa} and continued in \cite{RACSAM, AA_25}. The spaces ${\mathcal S}_\omega ({\mathbb R}) $ are versions of the Schwartz class in the ultradifferentiable setting, the most important examples being the classical Gelfand-Shilov spaces $\Sigma_s,$ $s>1.$

Both the Schwartz class and Gelfand\--Shilov spaces can be obtained as intersections of weighted modula\-tion and ultra\--modula\-tion spaces. It is therefore natural to consider composition operators on weighted (ultra-)modula\-tion spaces. In 1953, Beurling and Helson proved that if $\psi$ is continuous on the torus and the corresponding  composition operator is continuous on the algebra of absolutely convergent Fourier series, then $\psi(t)=kt+\psi(0)$ for some integer $k.$ Later, Lebedev and Olevskii \cite{LO_94} showed that if $\psi:{\mathbb R}^d \to {\mathbb R}^d$ is ${\mathcal C}^1$ and defines a continuous composition operator from ${\mathcal F}L^p({\mathbb R}^d)$  into itself for $1\leq p <\infty$, $p\neq 2,$ then $\psi$ must be an affine bijection. More recently, Oukoudjou \cite{okoudjou} (see also \cite{RSTT}) extended this result to modulation spaces $M^{p,q}({\mathbb R}^d)$ under the assumption $2\neq q <\infty.$
\par 
Our aim in this paper is to establish sufficient conditions on a smooth function $\psi:{\mathbb R}^d \to {\mathbb R}^d$ to ensure that the composition operator is well-defined and continuous between two weighted (ultra-)modulation spaces. Using the Fourier inversion formula, one has $$f(\psi(x))=\int_{{\mathbb R}^d}\widehat{f}(\xi)e^{2\pi i \psi(x) \xi}d\xi, \, f \in {\mathcal S}({\mathbb R}^d),$$ which suggest studying Fourier integral operators with symbol identically equal to 1 and phase $\Phi(x,\xi) = \psi(x)\xi.$ Fourier integral operators on modulation spaces have been investigated  in a series of papers (see \cite{Cordero_2012_approximation, FGP2019} and the references therein). The conditions imposed on the phase imply that: 
\begin{enumerate}
    \item $|\partial_x^\alpha\psi(x)|\leq C_\alpha,$\ $|\alpha|\geq 1, x\in {\mathbb R}^d.$
    \item $|\mbox{det} \left(\frac{\partial \psi_j}{\partial x_i}(x)\right)| \geq \delta > 0,$ $x\in {\mathbb R}^d.$
\end{enumerate}
The second condition,  known as the non-degeneracy condition, imposes significant restrictions for our purposes. Instead, we represent the composition operator as a Kohn-Nirenberg pseudodifferential  operator $$
\begin{array}{*2{>{\displaystyle}l}}
	f\left(\psi(x)\right) & = \int_{{\mathbb R}^d} \widehat{f}(y) e^{2\pi i y \psi(x)}\, dy \\ & \\ & = \int_{{\mathbb R}^d} \sigma(x,y)\widehat{f}(y) e^{2\pi i x y}\, dy = \sigma(x,D) f(x),
\end{array}$$ with symbol 
$$
\sigma(x,y) = \exp\left(i \phi(x)y\right),\ \phi(x) = 2\pi (\psi(x) - x).
$$ We thus study pseudodifferential operators whose symbols have the form $\sigma(x,y) = \exp\left(i \phi(x)y\right)$ under suitable assumptions on $\phi.$ As a consequence, we obtain continuity results for certain composition operators acting between different \mbox{(ultra-)} modulation classes. In the case $d=1,$ it is worth noting that \cite{jmaa} showed that composition with a polynomial of degree greater than one never leaves Gelfand-Shilov spaces invariant; however, positive results were obtained when $\psi$ satisfies $|x|\leq C(1+|\psi(x)|)$ and $\psi'\in {\mathcal B}_{L_\infty,\omega}$ (defined below, see \cite[Corollary 4.2]{jmaa}). Moreover, by \cite[Theorem 3.9]{jmaa}, if the composition with $\psi$ leaves $\Sigma_s({\mathbb R})$ invariant, then $\psi'$ must be bounded. This motivates the type of conditions we impose on $\phi.$ 

The paper is organized as follows. In  \ref{sec:tempered}, we study continuity results for certain  composition operators on weighted modulation spaces $M^p_m({\mathbb R}^d)$ for $1\leq p \leq \infty.$ Section \ref{sec:ultra} considers the analogous problem for ultramodulation spaces associated with subadditive weight functions in the sense of Braun, Meise and Taylor, in the one dimensional setting; the results are then extended to several variables.

\section{Preliminaries}\label{sec:preliminar}
\subsection{Classes of ultradifferentiable functions}
 \begin{defn}\label{weightfunction} A continuous increasing function $\omega :[0,\infty [\longrightarrow [0,\infty [$ is called a {\it subadditive  weight}, or shortly a weight,  if it satisfies:
	\begin{itemize}
		\item[$(\alpha)$]  $\omega (s+t) \leq \omega (s)+\omega (t)$ for
		all $s, \, t\geq 0$,
		\item[$(\beta)$] $\displaystyle\int_{0}^{\infty}\frac{\omega (t)}{1+t^{2}}\ dt < \infty $,
		\item[$(\gamma)$] $\log(1+t^{2})=o(\omega (t))$ as t tends to $\infty$,
		\item[$(\delta)$] $\varphi_\omega :t\rightarrow \omega (e^{t})$ is convex.
	\end{itemize}
\end{defn}
\par\medskip
$\omega$ is extended to ${\mathbb R}^d$ as $\omega(x) = \omega(|x|).$
\par\medskip 
The {\it Young conjugate} $\varphi_\omega ^{*}:[0,\infty [ \longrightarrow {\mathbb R}$  of
$\varphi_\omega$ is defined by
$$
\varphi_\omega^{*}(s):=\sup\{ st-\varphi_\omega (t):\ t\geq 0\},\ s\geq 0.
$$ 

The following inequality will be used throughout the paper:
\begin{equation}\label{Young_1}
3 m\varphi_\omega^{\ast}(\frac{j}{3m})+j\leq m\varphi_\omega^{\ast}(\frac{j}{m}) \mbox{ for all } j,\, m \in{\mathbb N}.
   \end{equation}
In fact, since $\varphi_\omega(x+1)=\omega (e e^x)\leq 3\omega(e^x)= 3\varphi_\omega(x),$ we have 
$$\begin{array}{c}3\varphi_\omega^{\ast}(\frac{j}{3m})+\frac{j}{m}= \sup_{x\geq 0}\{(x+1)\frac{j}{m}-3\varphi_\omega(x)\}\\ \\ \leq \sup_{x\geq 0}\{(x+1)\frac{j}{m}-\varphi_\omega(x+1)\}\leq \varphi_\omega^{\ast}(\frac{j}{m}). \end{array} $$ 

\begin{defn}\label{originalDefEomega}
       For a weight function $\omega$  we define the $\omega$-ultradifferentiable functions on ${\mathbb R}^d$  of Beurling type, $\mathcal{E}_{\omega}({\mathbb R}^d)$, as\begin{equation*}
       \begin{split}
          \mathcal{E}_{\omega}({\mathbb R}^d)=\{f\in C^\infty({\mathbb R}^d): \textit{ for each compact }K\textit{ and for each }m\in\mathbb{N}: \\
           p_{K,m}(f)<\infty\}
       \end{split}
   \end{equation*} where $ p_{K,m}(f)=\sup_{x\in K, \alpha\in\mathbb{N}_0^d} |f^{(\alpha)}(x)| \exp\left(-m\varphi^*_\omega(\frac{|\alpha|}{m})\right)$.
   \end{defn}

   $\mathcal{E}_{\omega}({\mathbb R}^d)$ endowed with the topology induced by the seminorms $p_{K,m}$ $ m\in {\mathbb N}$ and $K$ running in the compact sets of ${\mathbb R}^d$ is a Fr\'echet space. Condition ($\beta$) for the weight implies that $\mathcal{D}_{\omega}({\mathbb R}^d):=\mathcal{E}_{\omega}({\mathbb R}^d)\cap \mathcal{D}({\mathbb R}^d)$ is non trivial.

   The spaces  $\mathcal{B}_{L_\infty,\omega}({\mathbb R}^d) $ were defined  in \cite{BFG_2003} as extensions of the classical  space  $\mathcal{B}_{L_\infty} ({\mathbb R}^d) ,$  introduced by Schwartz, of bounded smooth functions having bounded derivatives. $\mathcal{B}_{L_\infty,\omega}({\mathbb R}^d)$  consists of those functions $f\in C^\infty ({\mathbb R}^d)  $ such that for every $m\in\mathbb{N}$ there exists $C_m$ satisfying \begin{equation} \label{def_B}
		\left|f^{(\alpha)}(x)\right|\leq C_m\hspace{0.05cm}\exp\left(m\varphi_\omega^*(\frac{|\alpha |}{m})\right)
	\end{equation} for all $\alpha \in\mathbb{N}_0^d$ and $x\in\mathbb{R}^d$. It is endowed with the topology defined by the sequence of seminorms $$\sup_{x \in {\mathbb R}^d}\sup_{\alpha \in {\mathbb N}_0^d} \left|f^{(\alpha)}(x)\right| \exp\left(-m\varphi_\omega^*(\frac{|\alpha |}{m})\right), \, m \in {\mathbb N}.$$ The convexity of $\varphi^{\ast}_\omega$ implies that we can replace   $m\varphi_\omega^*(\frac{|\alpha |}{m})$ by $m\varphi_\omega^*(\frac{|\alpha |+1}{m})$ in the definition of $\mathcal{B}_{L_\infty,\omega}({\mathbb R}^d) $  as well as in the seminorms that define its topology.
    

For a space of functions $E(\mathbb R^d),$ we put 
$$
E(\mathbb R^d,\mathbb R^d) = \left\{(\phi_1,\dots,\phi_d):\ \phi_i \in E(\mathbb R^d)\right\}.
$$

The spaces $\mathcal{S}_{\omega}(\mathbb{R}^d)$  were defined by Bj\"orck. We give an equivalent definition \cite{seminormas}.
   \begin{defn}
Let $\omega$ be a weight function. The Gelfand-Shilov space of Beurling type $\mathcal{S}_{\omega}(\mathbb{R}^d)$ consists of those functions $f\in C^\infty({\mathbb R}^d)$ such that  
$$
\pi_{\lambda, \mu}(f):= \sup_{\alpha \in {\mathbb N}_0^d }\sup_{x\in {\mathbb R}^d}|f^{(\alpha)}(x)|e^{-\lambda\varphi^\ast(\frac{|\alpha|}{\lambda}) + \mu\omega(x)} < \infty,\ \ \mbox{for all } \lambda > 0, \mu > 0.$$

\end{defn}

The classical Gelfand-Shilov spaces $\Sigma_s({\mathbb R}^d),$ $s>1,$ are defined by the weights $\omega(t)=t^{1/s}.$
\subsection{Time-frequency analysis}
Given $f, \, g \in L^2({\mathbb R}^d),$ $g\neq 0,$ the short time Fourier transform (STFT) of $f$ with respect to the window $g$ is defined as 

$$V_g f(x,\omega)=\displaystyle  \int _{{\mathbb R}^d} f (t) \overline{g(t-x)} e^{-2\pi i t \omega}\, dt.$$ 

The STFT satisfies the ortogonality relations \cite[3.2.1]{g} \begin{equation}\label{ortogonality}\langle V_{g_1}f_1, V_{g_2}f_2\rangle= \langle f_1,f_2\rangle \overline{\langle g_1,g_2\rangle }, \end{equation} where the brackets  stand for the inner product in $L^2({\mathbb R}^d).$ 

We will also use brackets $\langle f,g \rangle  $ to denote the extension to ${\mathcal S}'({\mathbb R}^d)\times {\mathcal S}({\mathbb R}^d)$ (or to ${\mathcal S}_{\omega}'({\mathbb R}^d)\times {\mathcal S}_{\omega}({\mathbb R}^d)$) of the inner product in $L^2({\mathbb R}^d).$ 

The modulation and translation operators are defined by $M_{\omega}f(t)=e^{2\pi \imath \omega t}f(t)$ and $T_{x}f(t)=f(t-x).$ For a non-zero $g \in {\mathcal S}({\mathbb R}^d)$ (resp. $g \in {\mathcal S}_{\omega}({\mathbb R}^d)$) and a tempered distribution $f\in {\mathcal S}'({\mathbb R}^d)$ (resp. $f\in {\mathcal S}_{\omega}'({\mathbb R}^d)$),  the STFT of $f$ with window $g$ is given by the continuous function 
$$
V_{g}f(x,\omega)=\left<f,M_{\omega}T_x g \right>.$$  
\medskip
A positive function $v$ on ${\mathbb R}^N,$ is said to be submultiplicative if 
$$
v(x_1 + x_2) \leq v(x_1)v(x_2),\ \ x_1,x_2\in {\mathbb R}^N.$$ If $v$ is submultiplicative,  a positive function $m$ is called $v$-moderate if for some constant $C>0$,
$$
m(x_1 + x_2) \leq Cm(x_1)v(x_2),\ \ x_1,x_2\in {\mathbb R}^N.$$
\par
\medskip
Important examples of submultiplicative weights are the functions 
$$(1+|z|)^s\ \ (s\geq 0),\ v(z)=e^{a|z|},\ e^{ a \omega(z)}\ \ (a > 0),$$  when $\omega $ is a  weight function. 
The weights $m(z)=(1+|z|)^s$ $(s\in {\mathbb R})$ are $(1+|z|)^{|s|}$-moderate and the weights $e^{a\omega (z)},$ ($a\in {\mathbb R}$), are  $e^{|a| \omega(z)}$-moderate. 
\par\medskip 
Let $m$ be a  $(1+|z|)^s$-moderate weight on ${\mathbb R}^{2d}$ for some $s> 0$.
 We denote by $M_{m}^{p}({\mathbb R}^d)$ the space
$$
M_{m}^{p}({\mathbb R}^d):=\{ f \in {\mathcal S}'({\mathbb R}^d):\, m  \, V_{g}f \in L^{p}({\mathbb R}^{2d}) \}$$ endowed with the norm $$\|f\|_{M_{m}^{p}}=\|m \, V_{g}f \|_{L^{p}}.$$ 

If $m\equiv 1$, we simply write $M^p({\mathbb R}^d).$ $M^{p}_m({\mathbb R}^d)$  is a Banach space, its definition is independent of the choice of $g$ and different windows give equivalent norms.  In the case $1\leq p < \infty,$ the Schwartz class ${\mathcal S}({\mathbb R}^d)$ is dense in $M^{p}_m({\mathbb R}^d).$ By  $M^0_m({\mathbb R}^d)$ we denote the closure of  ${\mathcal S}({\mathbb R}^d)$ in  $M^{\infty }_m({\mathbb R}^d).$ For $1\leq p<\infty$ or $p=0$  the  dual  of $M^{p}_m({\mathbb R}^d)$ can be identified with $M^{p'}_{\frac{1}{m}}({\mathbb R}^d)$ where $\displaystyle \frac{1}{p}+\frac{1}{p'}=1.$   

If $m$ is only time dependent, that is $m (x, \zeta)=m (x)$, then 
$$
M^2_{m}({\mathbb R}^d)=L^2_{m}({\mathbb R}^d):=\{f : \, \, m \,f \in L^2({\mathbb R}^d) \}$$ whereas for $m (x, \zeta)=m( \zeta)$, the modulation space is 
$$
\{ f\in {\mathcal S}'({\mathbb R}^d): \, m \hat{f} \in L^2({\mathbb R}^d) \}.$$ In particular, the Sobolev spaces $H^s({\mathbb R}^d)$ may be obtained in this way.

 Moreover, by \cite{GZ_2004},  $$
{\mathcal S}({\mathbb R}^d) = \bigcap_{s > 0}M^\infty_{v_s}({\mathbb R}^d).$$ 

Finally, to work with ultradistributions in ${\mathcal S}'_{\omega}({\mathbb R^d})$ we fix  a non-zero function $g \in{\mathcal S}_{\omega}({\mathbb R^d}).$ Given   $a>0,$ $1\leq p\leq \infty$     and  a function $m:{\mathbb R}^{2N}\to [0,\infty[$  which is $e^{a\omega}$-moderate,  the  modulation space $M_{m}^{p}({\mathbb R}^d)$ consists of all $f\in {\mathcal S}'_{\omega}({\mathbb R^d})$  such that $m V_{g}f \in L^{p}({\mathbb R}^{2d})$ with the norm  $$\|f \|_{M_{m}^{p}}=\|m \, V_{g}f \|_{L^{p}}.$$ We will refer to them as modulation spaces of ultradistributions or ultramodulation spaces. Again,  $M_{m}^{p}({\mathbb R}^d)$  is a Banach space and its definition is independent of the choice of the window $g.$ For $p=2$, $M^2_{m}$ is a Hilbert space. 
${\mathcal S}_{\omega}({\mathbb R^d})$ is dense in $M_{m}^{p}({\mathbb R}^d)$ whenever $1\leq p <\infty$ and $M_{m}^{0}({\mathbb R}^d)$ denotes the closure of  ${\mathcal S}_{\omega}({\mathbb R^d})$ in $M_{m}^{\infty}({\mathbb R}^d).$
\par\medskip 
The dual of $M_{m}^{p}({\mathbb R}^d)$ for $1\leq p< \infty$  or $p=0$ can be identified with the modulation space $M_{\frac{1}{m}}^{p'}({\mathbb R}^d)$ where $\frac{1}{p}+\frac{1}{p'}=1.$
\par\medskip 
The ortogonality relations (\ref{ortogonality}) extend to the dual pairs $$\left(M_{m}^{p}({\mathbb R}^d), M_{\frac{1}{m}}^{p'}({\mathbb R}^d)\right).$$

Given a subadditive weight  $\omega$, let us write $\nu_s(z)=e^{s \omega(z)},$ then 
\begin{equation}\label{Gelfand-Shilov}
\bigcap_{s>0}M^\infty_{\nu_s}({\mathbb R}) = {\mathcal S}_{\omega}({\mathbb R}).\end{equation}
We refer to \cite{cr,g} for the necessary background on modulation spaces.

\section{Modulation spaces of tempered distributions}\label{sec:tempered}
\medskip
Let $\psi:{\mathbb R}^d\to {\mathbb R}^d$ be a smooth function. Recall that the composition operator $C_\psi:{\mathcal S}({\mathbb R}^d)\to {\mathcal S}({\mathbb R}^d),$ $f\mapsto f\circ \psi,$ can be represented as a Kohn-Nirenberg pseudodifferential operator:
$$
	f\left(\psi(x)\right) = \sigma(x,D) f(x),$$ where 
    $$\sigma(x,y) = \exp\left(i \phi(x)y\right),\ \phi(x) = 2\pi (\psi(x) - x).$$ 

According to \cite[page 192 and Proposition 4.3.1]{cr}, the Kohn-Nirenberg operator $\sigma(x,D):{\mathcal S}({\mathbb R}^d)\to {\mathcal S}^\prime({\mathbb R}^d)$ with symbol $\sigma\in {\mathcal S}^\prime({\mathbb R}^{2d})$ satisfies 
   $$
\langle \sigma(x,D)f, g\rangle = \langle \sigma, R(g,f)\rangle,\ \ f,g\in {\mathcal S}({\mathbb R}^d).$$ Here, $R(g,f)$ denotes the Rihaczek distribution and it is defined by 
$$
R(g,f)(x,\xi):= e^{-2\pi i x \xi} g(x) \overline{\widehat{f}(\xi)},\ x,\xi\in {\mathbb R}^d.$$

In order to study the continuity of $\sigma(x,D)$ between appropriate modulation spaces we need information about the short-time Fourier transform of $\sigma$ and $R(g,f).$ This is so because  using the orthogonality relations, fixing $\Phi \in {\mathcal S}({\mathbb R}^d)\setminus\{0\},$ with $\|\Phi\|_2 = 1,$ we have 
$$
\langle \sigma(x,D)f, g\rangle = \langle \sigma, R(g,f)\rangle = \langle V_{\Phi}\sigma, V_{\Phi} R(g,f)\rangle \,\ \ f,g\in {\mathcal S}({\mathbb R}^d).$$ 

\par\medskip 
With respect to $R(g,f)$, according to \cite[Lemma 1.3.39]{cr}, for $f,g,\ $ $\varphi_1, \varphi_2\in {\mathcal S}({\mathbb R}^d)$ and $\Phi_0 = R(\varphi_1, \varphi_2)\in {\mathcal S}({\mathbb R}^{2d})$ one has
$$
V_{\Phi_0}\left(R(g,f)\right)(z,\zeta) = e^{-2\pi i z_2 \zeta_2}V_{\varphi_1}g(z_1, z_2+\zeta_1)\overline{V_{\varphi_2}f(z_1 + \zeta_2, z_2)},$$ where $z = (z_1, z_2), \zeta = (\zeta_1, \zeta_2)\in {\mathbb R}^{2d}.$ \par\medskip Consequently, if $\|\varphi_1\|_2 = \|\varphi_2\|_2 = 1$ then 
$$
\left|\langle \sigma(x,D)f, g\rangle\right| \leq \int_{{\mathbb R}^{4d}}|V_{\Phi_0}\sigma(z,\zeta)|\cdot |V_{\varphi_1}g(z_1,z_2+\zeta_1)|\cdot |V_{\varphi_2}f(z_1+\zeta_2,z_2)|\, d(z,\zeta).
$$

All relations above, with the appropriate changes,  extend to the ultradistibutional setting, therefore we will use them later.
\par\medskip
The information we need about  $\sigma$ is collected in the following result. From now on, we write  $A(x)\lesssim B(x)$ meaning that there exists $C>0$ such that $A(x)\leq C B(x)$ for all $x.$ 

\begin{lem}\label{lem:sigma} Let $\phi \in C^\infty({\mathbb R}^d, {\mathbb R}^d)$ be given and $\sigma(x,y) = \exp(i\phi(x) y)$. 
\begin{enumerate}
\item  If $\partial_{x_j} \phi\in {\mathcal B}_{L^\infty}(\mathbb{R}^d, \mathbb{R}^d)$ for each $1\leq j \leq d,$  then, for any  $N\in {\mathbb N},$ we have
	$$
	\sigma\in  M^\infty_{m}({\mathbb R}^{2d})$$ where $m(z,\zeta) = \left(1+|\zeta_1|\right)^N \left(1+|z_2|\right)^{
		-N}.$

        \item Let  $\phi:\mathbb{R}^d\to\mathbb{R}^d$ satisfy $|\phi(x)| = O(|x|^b)$ as $|x|\to \infty$ for some $b\geq 0.$  Then, for any $N\in {\mathbb N}$ we have $\sigma\in M^\infty_{m}({\mathbb R}^{2d})$ where 
	$$m(z,\zeta) = \left(1+|\zeta_2|\right)^{N}(1 + |z_1|)^{-Nb}.$$ \end{enumerate} 
\end{lem}
\begin{proof} \begin{enumerate}
  \item Given a  multi index $\beta\in {\mathbb N}_0^d$ we have
$$
	\begin{array}{*1{>{\displaystyle}l}}
		|\zeta_1^\beta|	\left|\left(V_{\varphi\otimes\varphi}\sigma\right)(z,\zeta)\right|  \lesssim \int_{{\mathbb R}^{2d}} |\partial_x^\beta\left(\sigma(x,y)\varphi(x-z_1)\right)	|\varphi(y-z_2)|\, d(x,y) \\  \\  \lesssim \sum_{\gamma\leq \beta}\int_{{\mathbb R}^{2d}}|\varphi^{(\gamma)}(x-z_1)|\ dx \cdot \int_{{\mathbb R}^{2d}} (1+|y|)^{|\beta-\gamma|} |\varphi(y-z_2)|\ d(x,y)\\  \\  \leq \sum_{\gamma\leq \beta}\|\varphi^{(\gamma)}\|_1\cdot \int_{{\mathbb R}^{d}}\left(1+|y| + |z_2|\right)^{|\beta-\gamma|} |\varphi(y)|\ dy\\  \\  \leq \sum_{\gamma\leq \beta}\|\varphi^{(\gamma)}\|_1 \left(1+|z_2|\right)^{|\beta-\gamma|}\int_{{\mathbb R}^{d}}\left(1+|y|\right)^{|\beta-\gamma|} |\varphi(y)|\ dy\\  \\  \lesssim \left(1 + |z_2|\right)^{|\beta|}.
	\end{array}$$ 

    Since $\left(1 + |\zeta_1|\right)^N \lesssim \sum_{|\beta|\leq N}|\zeta_1^\beta|$ we finally conclude
	$$
	\left(1 + |\zeta_1|\right)^N\left|\left(V_{\varphi\otimes\varphi}\sigma\right)(z,\zeta)\right| \leq C \left(1 + |z_2|\right)^{N}$$ for some constant $C > 0.$

\item Let us fix a non-zero smooth function $\varphi$ on ${\mathbb R}^d$ supported on $[-1,1]^d.$ Then 
$$
\begin{array}{*1{>{\displaystyle}c}}
V_{\varphi\otimes\varphi}\sigma(z,\zeta) = \int_{{\mathbb R}^{2d}}\sigma(x,y)\varphi(x-z_1)\varphi(y-z_2)e^{-2\pi i (x\zeta_1 + y\zeta_2)}\, d(x,y)\\  \\  = \int_{{\mathbb R}^{d}}\varphi(x-z_1)e^{-2\pi i x\zeta_1}\left(\int_{{\mathbb R}^{d}}\varphi(y-z_2)e^{-2\pi i y(\zeta_2-\phi(x))}\, dy\right)\, dx  \\  \\  = \int_{{\mathbb R}^{d}}\varphi(x-z_1)e^{-2\pi i x\zeta_1}e^{-2\pi i z_2(\zeta_2-\phi(x))}\widehat{\varphi}(\zeta_2 - \phi(x))\, dx.
\end{array}$$ 
We have
$$
\begin{array}{*2{>{\displaystyle}l}}
\left|V_{\varphi\otimes\varphi}\sigma(z,\zeta)\right| & \leq \int_{{\mathbb R}^{d}}|\varphi(x-z_1)| |\widehat{\varphi}(\zeta_2 - \phi(x))|\, dx	\\ & \\ & \lesssim \int_{|x-z_1|\leq 1}|\varphi(x-z_1)|\frac{dx}{(1+|\zeta_2 - \phi(x)|)^N} \\ & \\ & \lesssim \int_{|x-z_1|\leq 1}|\varphi(x-z_1)|\frac{(1+|\phi(x)|)^N}{(1+|\zeta_2|)^N}\, dx\\ & \\ & \lesssim \frac{(1+|z_1|)^{Nb}}{(1+|\zeta_2|)^N}.
\end{array}$$ 
\end{enumerate}
\end{proof}

\begin{rem}{\rm  
In the case that $\phi$ is bounded we take $b = 0$ and we obtain that for any $N\in {\mathbb N},$ $\sigma\in M^\infty_{m}({\mathbb R}^{2d})$ where $m(z,\zeta) = \left(1+|\zeta_2|\right)^{N}.$
 }
\end{rem}

\par\medskip
In what follows we fix a non-zero window $\varphi\in {\mathcal S}({\mathbb R}^d)$ with $\|\varphi\|_2 = 1$ and take $\Phi_0 = R(\varphi, \varphi).$ All the norms in modulation spaces use one of these windows, depending on the dimension.
\par\medskip 
\begin{thm}\label{cont_psdos}
Let us assume $\partial_{x_j}  \phi\in {\mathcal B}_{L^\infty}(\mathbb{R}^d, \mathbb{R}^d)$ for each $1\leq j \leq d$ and $|\phi(x)| = O(|x|^b)$ as $|x|\to \infty$ for some  $b\in [0,1).$ Let $\sigma(x,y) = \exp(i\phi(x) y).$ Given $s\geq  0$ we take $N\in {\mathbb N}$ such that $(1-b)N > 2(s+d) $ and consider the weights $v_s(z) = (1+|z|)^s$ and $m(z) = v_{s+Nb}(z)(1+|z_2|)^N,$ $z = (z_1, z_2)\in {\mathbb R}^{2d}.$ Then, for every $1\leq p\leq \infty,$
	$$
	\sigma(x,D):M^p_m({\mathbb R}^d)\to M^p_{v_s}({\mathbb R}^d)$$ is continuous.  
\end{thm}
\begin{proof}
By interpolation (see \cite[Proposition 2.3.16]{cr}) it suffices to prove the result for $p= 1$ and $p=\infty.$
\par\medskip 
We first consider the case $p = 1.$ For any $f,g\in {\mathcal S}({\mathbb R}^d)$ we have that $\left|\langle \sigma(x,D)f, g\rangle\right|$ is less than or equal to 
 $$
	\begin{array}{*1{>{\displaystyle}l}}
		 \|g\|_{M^\infty_{\frac{1}{v_s}}} \int_{{\mathbb R}^{4d}} \left|\left(V_{\Phi_0}\sigma\right)(z,\zeta)\right| v_s(z_1, z_2+\zeta_1) \left|V_\varphi f(z_1+\zeta_2, z_2\right|\, d(z, \zeta)\\  \\  \leq \|g\|_{M^\infty_{\frac{1}{v_s}}}\ 
        \int_{{\mathbb R}^{4d}}\left|\left(V_{\Phi_0}\sigma\right)(z,\zeta)\right| \frac{v_s(z_1, z_2+\zeta_1)}{(1+|z_2|)^{N}} \frac{|F(z_1 + \zeta_2, z_2)|}{v_{s+Nb}(z_1+\zeta_2, z_2)}\, d(z,\zeta)
	\end{array}$$ where 
    $$
    F(z_1,z_2) = (V_\varphi f)(z_1,z_2)m(z_1,z_2).
    $$   
\noindent
    Having in mind that $v_t$ and $v_t^{-1}$ are $v_t$-moderate for every $t > 0$ we have that
	$$
	\left|\left(V_{\Phi_0}\sigma\right)(z,\zeta)\right| (1 + |z_2|)^{-N} v_{s+Nb}^{-1}(z_1+\zeta_2, z_2)v_s(z_1, z_2 + \zeta_1)
	$$ is less than or equal to some constant times
	$$
    \begin{array}{*1{>{\displaystyle}l}}
	\left|\left(V_{\Phi_0}\sigma\right)(z,\zeta)\right| (1 + |z_2|)^{-N} v_{Nb}^{-1}(z_1, z_2)v_{s+Nb}(\zeta_2,0) v_s(0, \zeta_1)\\ \\ \leq \left|\left(V_{\Phi_0}\sigma\right)(z,\zeta)\right| (1 + |z_2|)^{-N} (1+|z_1|)^{-Nb} v_{s+Nb}(\zeta_2,0) v_s(0, \zeta_1).
    \end{array}
    $$ From Lemma \ref{lem:sigma}, for every  $N > 0$ we have
	$$
	\left|\left(V_{\Phi_0}\sigma\right)(z,\zeta)\right| (1 + |z_2|)^{-N} (1 + |z_1|)^{-Nb}\lesssim \min\left((1+|\zeta_1|)^{-N}, (1+|\zeta_2|)^{-N} \right).$$ Hence
	$$
	\left|\langle \sigma(x,D)f, g\rangle\right|	\lesssim \|f\|_{M^1_m} \|g\|_{M^\infty_{\frac{1}{v_s}} }\int_{{\mathbb R}^{2d}} G(\zeta)\ d\zeta$$ where
	$$
	G(\zeta) = \min\left((1+|\zeta_1|)^{-N}, (1+|\zeta_2|)^{-Nb} \right) v_{s+Nb}(\zeta_2, 0) v_s(0, \zeta_1).$$ Since
	$$
	\int_{|\zeta_1| < |\zeta_2|} d\zeta_1 \lesssim |\zeta_2|^d$$ we get
	$$
	\begin{array}{*2{>{\displaystyle}l}}
		\int_{|\zeta_1| < |\zeta_2|} G(\zeta)\ d\zeta & \leq \int_{|\zeta_1| < |\zeta_2|}\frac{v_{2s+Nb}(\zeta_2,0)}{(1+|\zeta_2|)^N}\ d\zeta \\ & \\ & \lesssim \int_{{\mathbb R}^{d}} \frac{|\zeta_2|^dv_{2s+Nb}(\zeta_2,0)}{(1+|\zeta_2|)^N}\ d\zeta_2	
	\end{array}$$     and 
	$$
	\int_{|\zeta_2| < |\zeta_1|} G(\zeta)\ d\zeta \lesssim \int_{{\mathbb R}^{d}}\frac{|\zeta_1|^dv_{2s+Nb}(0, \zeta_1)}{(1+|\zeta_1|)^N}\ d\zeta_1 .$$  Both integrals converge since $N > 2s+2d + Nb.$ Moreover
    $$
    \int_{{\mathbb R}^{2d}} \left|F(z_1 + \zeta_2, z_2)\right|\, dz = \|f\|_{M^1_m}.
    $$ We finally conclude 
    $$
    \left|\langle \sigma(x,D)f, g\rangle\right|	\leq C \|g\|_{M^\infty_{\frac{1}{v_s}}}\|f\|_{M^1_m}
    $$ where $C>0$ does not depend neither on $g$ nor on $f.$ This means that $\sigma(x,D)f$ defines a continuous linear form on $M_{\frac{1}{v_s}}^0({\mathbb R}^d),$ that is (see \cite[p. 92]{cr}) $\sigma(x,D)f \in M^1_{v_s}({\mathbb R}^d)$  and $\|\sigma(x,D)f\|_{M^1_{v_s}}\leq C\|f\|_{M^1_m}.$ Thus, $$\sigma(x,D):{\mathcal S}({\mathbb R}^d) \to {\mathcal S}'({\mathbb R}^d)$$ can be extended as a continuous operator $$\sigma(x,D):M^1_m({\mathbb R}^d) \to M^1_{v_s}({\mathbb R}^d). $$

    Regarding the case $p=\infty, $ for $f,g\in {\mathcal S}({\mathbb R}^d)$ we have 
    that $\left|\langle \sigma(x,D)f, g\rangle\right|$ is less than or equal to 
        $$
        \begin{array}{*1{>{\displaystyle}l}}
		\|f\|_{M^\infty_m} \int_{{\mathbb R}^{4d}}\left|\left(V_{\Phi_0}\sigma\right)(z,\zeta)\right| (1 + |z_2|)^{-N}  \frac{\left|V_\varphi g(z_1, z_2+\zeta_1)\right|}{v_{s+Nb}(z_1+\zeta_2, z_2)}\, d(z, \zeta)\\  \\  \leq \|f\|_{M^\infty_m}\ 
        \int_{{\mathbb R}^{4d}}\left|\left(V_{\Phi_0}\sigma\right)(z,\zeta)\right|  \frac{v_s(z_1, z_2 + \zeta_1)}{(1 + |z_2|)^{N}} \frac{|F(z_1, z_2+\zeta_1)|}{v_{s+Nb}(z_1+\zeta_2, z_2)}\, d(z,\zeta)
	\end{array}$$ where
	$ F = v_s^{-1} V_\varphi g .$ It turns out that
	$$
	\int_{{\mathbb R}^{2d}}\left|F(z_1, z_2+\zeta_1)\right|	dz = \int_{{\mathbb R}^{2d}}\left|F(z_1, z_2)\right|	dz = \|g\|_{M^1_{1/v_s}}$$ is independent of $\zeta_1.$   As in the case $p=1,$
    $$
    \left|\left(V_{\Phi_0}\sigma\right)(z,\zeta)\right| v_s(z_1, z_2+\zeta_1) v_{s+Nb}^{-1}(z_1+\zeta_2, z_2)(1+|z_2|)^{-N} \lesssim G(\zeta),
    $$ where 
    $$
    \int_{{\mathbb R}^{2d}} G(\zeta)\, d\zeta < \infty.
    $$ Whence, there is $C>0$ such that $$
	\left|\langle \sigma(x,D)f, g\rangle\right| \leq C \|f\|_{M^\infty_m} \|g\|_{M^1_{1/v_s}}$$ for every $f,g\in {\mathcal S}({\mathbb R}^d).$	Since we already know that  $\sigma(x,D)({\mathcal S}({\mathbb R}^d))\subset M^1_{v_s}\subset  M^0_{v_s}, $ this implies that $\sigma(x,D)$ extends to a continuous operator $$\sigma(x,D):M^0_m({\mathbb R}^d) \to M^0_{v_s}({\mathbb R}^d).$$ Taking the bi-transpose map, $$\sigma(x,D):M^\infty_m({\mathbb R}^d) \to M^\infty_{v_s}({\mathbb R}^d)$$ is continuous.
\end{proof}
\par\medskip 
As a consequence, we obtain the following result concerning the continuity of composition operators between weighted modulation spaces, with control of  the loss of regularity, in the case where $\psi$ is a perturbation of an affine map.

\begin{cor} Let us assume  that $\psi(x)=Ax+\phi(x),$ $A$ being an invertible $d\times d$ matrix,  $\partial_{x_j}  \phi\in {\mathcal B}_{L^\infty}(\mathbb{R}^d, \mathbb{R}^d)$ for each $1\leq j \leq d$ and $|\phi(x)| = O(|x|^b)$ as $|x|\to \infty$ for some  $b\in [0,1).$  Given $s \geq  0$ we take $t= \frac{1+b}{1-b} [2s] + s + 2d + 1.$ Then, for  $1\leq p \leq \infty,$ the composition operator 
	$$
	C_\psi:M^p_{v_t}({\mathbb R}^d)\to M^p_{v_s}({\mathbb R}^d)$$ is continuous.

    In particular, $$C_\psi:{\mathcal S}({\mathbb R}^d)\to {\mathcal S}({\mathbb R}^d) $$ is continuous.
\end{cor}
\begin{proof} Since the modulation spaces $M^p_{v_s}$ are invariant under  linear changes of coordinates, we may assume that $A$  is the identity. The  conclusion then follows from our previous result. 

Finally,  the continuity of $C_\psi$ in  $
{\mathcal S}({\mathbb R}^d)$ follows from 
$${\mathcal S}({\mathbb R}^d) = \bigcap_{s > 0}M^\infty_{v_s}({\mathbb R}^d).$$
\end{proof}

\section{Ultra-modulation spaces}\label{sec:ultra}
To avoid technicalities, we begin by considering the one variable case. Continuity results for composition operators on multivariable modulation spaces then follow as a consequence.

\subsection{One variable}
Our goal is to establish boundedness results within the framework of ultra\--modulation spaces. A key tool in this analysis will be the class of $\omega$-ultra\-differential operators, whose definition and fundamental properties we review below.

\begin{defn}\label{def:G(D)}
Suppose that $G\in H({\mathbb C})$ satisfies for some $m > 0$ 
\begin{equation}\label{eq:G}
\log|G(z)|\leq m(1+\omega(z)),\ z\in {\mathbb C}.\end{equation} Then we define an ultradistribution $T_G\in {\mathcal E}^\prime_\omega({\mathbb R})$ supported at the origin by
$$
\langle T_G, \varphi \rangle = \sum_{n=0}^\infty (-i)^n \frac{G^{(n)}(0)}{n!} \varphi^{(n)}(0),\ \ \varphi\in {\mathcal E}_\omega({\mathbb R}).
$$ The operator 
$$
G(D):{\mathcal D}^\prime_\omega({\mathbb R})\to {\mathcal D}^\prime_\omega({\mathbb R}),\ G(D)\mu = T_G\ast \mu,
$$ is called ultradifferential operator of class $\omega.$
\end{defn} 
It turns out that $G(D):{\mathcal E}_\omega({\mathbb R})\to {\mathcal E}_\omega({\mathbb R})$ is given by 
$$
G(D)f(x) = \sum_{n=0}^\infty i^n  \frac{G^{(n)}(0)}{n!}f^{(n)}(x).$$
\par\medskip 
The convergence of the previous series is justified by the following well known estimates. We include a proof for the convenience of the reader.

\begin{lem}
  Let $G\in H({\mathbb C})$ be as in Definition \ref{def:G(D)}. Then 
  $$
  \left|\frac{G^{(n)}(0)}{n!}\right| \leq e^m e^{-m\varphi_\omega^\ast\left(\frac{n}{m}\right)}, \ \ n\in {\mathbb N}_0.
  $$
\end{lem}
\begin{proof}
 For every $t>0$ and $r = e^t,$ the Cauchy inequalities give us   
 $$
 \left|\frac{G^{(n)}(0)}{n!}\right| \leq r^{-n}\exp\left(m(1+\omega(r))\right) = e^m\exp\left(m\left(\varphi_\omega(t)-\frac{n}{m}t\right)\right).
$$ Taking the infimum when $t>0$ we get the conclusion.
\end{proof}

We observe that $G(D_x)\left(e^{i\xi x}\right) = G(-\xi)e^{i\xi x}.$ Also, for any  $f,g\in {\mathcal E}_{\omega}({\mathbb R}),$ one of them compactly supported, we have
$$
\int_{{\mathbb R}} f(x) \left(G(D)g\right)(x)\, dx = \int_{{\mathbb R}} \left(G(-D)f\right)(x) g(x)\, dx.$$
\par\medskip 
According to \cite[Theorem 1]{braun}, for every $N\in {\mathbb N}$ there exists $G\in H({\mathbb C})$  such that $\log |G(z)| = O(\omega(z))$ as $|z|\to \infty$ and
\begin{equation}\label{eq:2}
\log |G(x)| \geq N \omega(x)\qquad \forall x\in {\mathbb R}.
\end{equation} 

In this case, $G(D)$ is said to be \emph{strongly elliptic}.

\begin{prop}\label{lem:G(D)-product} Let $G(D)$ be an ultradifferential operator of class $\omega.$ There exists $m_0\in {\mathbb N}$ such that the following holds. For every $g\in {\mathcal B}_{L_\infty, \omega}({\mathbb R})$ and $h\in {\mathcal E}_{\omega}({\mathbb R})$ we have
	$$
	G(-D)\left(gh\right)(x) = \sum_{k=0}^\infty a_k(x) h^{(k)}(x)$$ where the functions $\left(a_k(x)\right)_{k\in {\mathbb N}_0}$ depend on $g$ and satisfy
		$$
	\forall \ell\in {\mathbb N}\ \exists C_\ell > 0: \sup_{x\in {\mathbb R}}|a_k^{(r)}(x)| \leq C_\ell  \exp\left(-m_0\varphi_\omega^\ast(\frac{k}{m_0})+\ell\varphi_\omega^\ast(\frac{r}{\ell})\right)$$ for every $k, r\in {\mathbb N}_0,$ and the constant $C_\ell$ depends only on $\|g\|_{2\ell, \infty, \omega}.$
\end{prop}
\begin{proof} We denote $b_n = (-i)^n \frac{G^{(n)}(0)}{n!}$ and observe that 
	$$
	\begin{array}{*2{>{\displaystyle}l}}
	G(-D)(gh)(x) & = \sum_{n=0}^\infty b_n (gh)^{(n)}(x) = \sum_{n=0}^\infty b_n\left(\sum_{k=0}^n \binom{n}{k}h^{(k)}(x)g^{(n-k)}(x)\right)\\ & \\ & = \sum_{k=0}^\infty a_k(x) h^{(k)}(x),
	\end{array}$$ where
	$$
	a_k(x) = \sum_{n\geq k}\binom{n}{k} b_n g^{(n-k)}(x) = \sum_{j=0}^\infty \binom{k+j}{k} b_{k+j} g^{(j)}(x).$$ In order to justify the change in the order of summation in the expression for $G(-D)(gh)(x)$ we will check that, for every compact set $K\subset {\mathbb R},$
    \begin{equation}\label{eq:change-order}
    \sup_{x\in K}\sum_{n=0}^\infty |b_n|\sum_{k=0}^n \binom{n}{k}|h^{(k)}(x)|\cdot |g^{(n-k)}(x)| < \infty.\end{equation} In fact, since 
    $$
    \sup_{x\in {\mathbb R}}|g^{(j)}(x)| \leq \|g\|_{3m,\infty,\omega}\exp(3m\varphi^\ast_\omega(\frac{j}{3m})
    $$ and there exist $C_K>0$ such that  
    $$
    \sup_{x\in K}|h^{(j)}(x)|\leq C_K\exp(3m\varphi^\ast_\omega(\frac{j}{3m})
    $$ for every $j\in {\mathbb N}_0,$ we have that the supremum in (\ref{eq:change-order}) is less than or equal to a constant times 
    $$
    \begin{array}{*1{>{\displaystyle}l}}
     \sum_{n=0}^\infty \sum_{k=0}^n \binom{n}{k} \exp\left(-m\varphi^\ast_\omega(\frac{n}{m}) + 3m \varphi^\ast_\omega(\frac{k}{3m}) + 3m \varphi^\ast_\omega(\frac{n-k}{3m})\right) \\  \\  \leq\sum_{n=0}^\infty \sum_{k=0}^n \binom{n}{k}\exp\left(-m\varphi^\ast_\omega(\frac{n}{m}) + 3m \varphi^\ast_\omega(\frac{n}{3m})\right)  \leq \sum_{n=0}^\infty \sum_{k=0}^n \binom{n}{k}e^{-n} < \infty.
    \end{array} 
    $$
 To finish the proof, we estimate the derivatives of the functions $a_k(x).$  We fix $r\in {\mathbb N}_0$ and observe that, for every $\ell\in {\mathbb N},$
	$$
	\begin{array}{*2{>{\displaystyle}l}}
		\left|g^{(j+r)}(x)\right| & \leq \|g\|_{2\ell,\infty,\omega} \exp\left(2\ell \varphi^\ast_\omega(\frac{j+r}{2\ell})\right) \\ & \\ & \leq \|g\|_{2\ell,\infty,\omega}\exp\left(\ell\varphi^\ast_\omega(\frac{j}{\ell}) + \ell\varphi^\ast_\omega(\frac{r}{\ell})\right).
		\end{array}$$ Since
		$$
		\begin{array}{*2{>{\displaystyle}l}}
		|b_{k+j}| & \leq C\exp\left(-m\varphi^\ast_\omega(\frac{k+j}{m})\right) \\ & \\ & \leq C \exp\left(-m\varphi^\ast_\omega(\frac{j}{m}) - m \varphi^\ast_\omega(\frac{k}{m})\right)
		\end{array}$$ we conclude that
		$$
		\sum_{j=0}^\infty \binom{k+j}{k} |b_{k+j}| |g^{(r+j)}(x)|$$ is less than or equal to
		$$
		\|g\|_{2\ell,\infty,\omega}2^k \exp\left(\ell\varphi^\ast_\omega(\frac{r}{\ell}) - m \varphi^\ast_\omega(\frac{k}{m})\right)\sum_{j=0}^\infty 2^j\exp\left(-m\varphi^\ast_\omega(\frac{j}{m})+\ell\varphi^\ast_\omega(\frac{j}{\ell}\right).$$
		For every $\ell \geq 3m$ we have
		$$
		j + \ell\varphi^\ast_\omega(\frac{j}{\ell}) \leq  m\varphi^\ast_\omega(\frac{j}{m}).$$ Consequently
		$$
		\begin{array}{*1{>{\displaystyle}l}}
		\sum_{j=0}^\infty 2^j\exp\left(-m\varphi^\ast_\omega(\frac{j}{m})+\ell\varphi^\ast_\omega(\frac{j}{\ell}\right) \\ \\ =
		\sum_{j=0}^\infty \left(\frac{2}{e}\right)^j \exp\left(j -m\varphi^\ast_\omega(\frac{j}{m})+\ell\varphi^\ast_\omega(\frac{j}{\ell}\right) \leq \sum_{j=0}^\infty \left(\frac{2}{e}\right)^j =:D
		\end{array}$$ and we obtain
		$$
		|a_k^{(r)}(x)| \leq 2^k D \|g\|_{2\ell,\infty,\omega}\exp\left(-m \varphi^\ast_\omega(\frac{k}{m}) + \ell\varphi^\ast_\omega(\frac{r}{\ell})\right).$$ The conclusion follows with $m_0 = 3m$. 
\end{proof} 

\begin{lem}\label{lem:sigma-derivadas} Let $\sigma(x,y) = e^{i\phi(x)y}$ be given, where $\phi '\in {\mathcal B}_{L_\infty, \omega}({\mathbb R}).$ For every $\ell\in{\mathbb N}$ there exists $m_\ell > 0$ such that
$$
\left|\left(\partial_x^n\sigma\right)(x,y)\right| \leq 4^n e^{m_\ell\omega(y)} e^{\ell\varphi^\ast_\omega(\frac{n}{\ell})}\ \ \forall n\in {\mathbb N}.$$
\end{lem}
\begin{proof} We fix $\ell, n\in {\mathbb N}$ and denote $\displaystyle I:=\Big\{{\bm k} = (k_1, k_2, \ldots, k_n): \sum_{j=1}^n j k_j = n\Big\}.$ For every $\bm k\in I$ we put $k = \sum_{j=1}^n k_j.$ 
\par\medskip 
Since $$\|\phi'\|_{\ell,\infty,\omega} = \sup_{x\in {\mathbb R}}\sup_{k\in {\mathbb N}}|\phi^{(k)}(x)|e^{-\ell\varphi^\ast_\omega(\frac{k}{\ell})} < \infty$$ then we get, from Faà di Bruno formula, that there exist a constant $B_\ell > 0$ such that
$$
\begin{array}{*2{>{\displaystyle}l}}
\left|\left(\partial_x^n\sigma\right)(x,y)\right| & \leq \sum_{\bm k\in I}\frac{n!}{k_1!\ldots k_n!} |y|^k \prod_{j=1}^{n}\left|\frac{\phi^{(j)}(x)}{j!}\right|^{k_j}\\ & \\ & \overset{\text{\cite[p. 403]{superposition}}}{\leq}
\sum_{\bm k\in I}\frac{n!}{k_1!\ldots k_n!} |y|^k B_\ell^k \frac{\exp\left(\ell\varphi^\ast_\omega(\frac{n-k}{\ell})\right)}{(n-k)!}\\ & \\ & \leq 2^n \sum_{\bm k\in I}\frac{k!}{k_1!\ldots k_n!}\left(B_\ell |y|\right)^k\exp\left(\ell\varphi^\ast_\omega\left(\frac{n}{\ell}\right)-\ell\varphi^\ast_\omega\left(\frac{k}{\ell}\right)\right).
\end{array}$$ 
Using
$$
\left(B_\ell |y|\right)^k\exp\left(-\ell\varphi^\ast_\omega\left(\frac{k}{\ell}\right)\right) \leq \exp\left(\ell\omega(B_\ell y)\right)$$ we finally conclude
$$
\begin{array}{*2{>{\displaystyle}l}}
\left|\left(\partial_x^n\sigma\right)(x,y)\right| & \leq 2^n\exp\left(\ell\omega(B_\ell y)+\ell\varphi^\ast_\omega(\frac{n}{\ell})\right)\sum_{\bm k\in I}\frac{k!}{k_1!\ldots k_n!} \\ & \\ & = 2^{2n-1} e^{\ell\omega(B_\ell y)} e^{\ell\varphi^\ast_\omega(\frac{n}{\ell})}.
\end{array}$$ To finish the proof it suffices to take $m_\ell = \ell\lceil B_\ell\rceil.$	
\end{proof}

In the proof of Lemma \ref{lem:sigma-derivadas} we used the fact that
	
		$$\sum_{{\bm k}\in I}\frac{k!}{k_1!\ldots k_n!} = 2^{n-1},$$ as follows, for instance, after taking $h(x) = \frac{x}{1-x}$ and evaluating $(h\circ h)^{(n)}(0)$ with Faà di Bruno's formula.

\begin{lem}\label{lem:sigma-mod-2-previo} Let $\phi' \in {\mathcal B}_{L_\infty, \omega}({\mathbb R})$ be given, and let $\sigma(x,y) = e^{i\phi(x)y}.$ For every $N\in {\mathbb N}$ there is $k\in {\mathbb N}$ such that  
$$
\left\{e^{-k\omega(y)}\sigma(\cdot,y):\ y\in {\mathbb R}\right\}
$$ is a bounded set in $M^\infty_m({\mathbb R}),$ where $m(z_1,\zeta_1) = e^{N\omega(\zeta_1)}.$
\end{lem}
\begin{proof}
 We fix $\varphi\in {\mathcal D}_{\omega}([-1,1])$ and consider $G(D)$ as in (\ref{eq:G}) and (\ref{eq:2}). Then
$$
\begin{array}{*1{>{\displaystyle}l}}
G(2\pi \zeta_1)\left(V_{\varphi}\sigma(\cdot,y)\right)(z_1,\zeta_1) \\ \\ = \int_{{|x-z_1|\leq 1}} G(-D_x)\left(\sigma(x,y)\varphi(x-z_1)\right)e^{-2\pi i x\zeta_1 }\, dx.\end{array}$$ Since ${\mathcal E}_\omega({\mathbb R})$ is holomorphically closed (see for instance \cite{pv}) we have $\sigma(\cdot,y)\in {\mathcal E}_\omega({\mathbb R})$ and we can apply Proposition \ref{lem:G(D)-product} with $g(x) = \varphi(x-z_1)$ to get
$$
G(-D_x)\left(\sigma(x,y)\varphi(x-z_1)\right) = \sum_{k=0}^\infty a_k(x,z_1)\partial_x^k\left(\sigma(x,y)\right)$$ where, for some constants $C>0$ and $m_0\in {\mathbb N},$
$$
|a_k(x,z_1)| \leq C e^{-m_0\varphi^\ast_\omega(\frac{k}{m_0})}\ \ \forall k\in {\mathbb N}_0, x, z_1\in {\mathbb R}.$$   
 From Lemma \ref{lem:sigma-derivadas}, for every $\ell\in {\mathbb N}$ we have
$$
\left|G(-D_x)\left(\sigma(x,y)\varphi(x-z_1)\right)\right| \leq C \sum_{k=0}^\infty e^{-m_0\varphi^\ast_\omega(\frac{k}{m_0})} 4^k e^{m_\ell\omega(y)+\ell\varphi^\ast_\omega(\frac{k}{\ell})}.$$ We now fix $\ell \geq 9 m_0.$  From
$$
e^{-m_0\varphi^\ast_\omega(\frac{k}{m_0})} \leq e^{-2k} e^{-\ell\varphi^\ast_\omega(\frac{k}{\ell})}$$ we get
$$
\left|G(-D_x)\left(\sigma(x,y)\varphi(x-z_1)\right)\right| \leq C \sum_{k=0}^\infty \left(\frac{4}{e^2}\right)^k e^{m_\ell\omega(y)} =: De^{m_\ell\omega(y)}.$$ Consequently,
$$
e^{N\omega(\zeta_1)}\left|\left(V_{\varphi}\sigma(\cdot,y)\right)(z_1,\zeta_1)\right| \leq  2D e^{m_\ell\omega(y)}.
$$ The conclusion follows with $k=m_\ell.$	
\end{proof}

\begin{prop}\label{lem:sigma-mod-2} Let $\phi' \in {\mathcal B}_{L_\infty, \omega}({\mathbb R})$ be given, and let $\sigma(x,y) = e^{i\phi(x)y}.$ For every $N\in {\mathbb N}$ there is $k\in {\mathbb N}$ such that  $\sigma\in M^\infty_m({\mathbb R}^2),$ where $m(z,\zeta) = e^{N\omega(\zeta_1)-k\omega(z_2)}.$
\end{prop}
\begin{proof}
We fix $\varphi\in {\mathcal D}_{\omega}([-1,1]).$ It follows from the definitions that 
$$
\left(V_{\varphi\otimes\varphi}\sigma\right)(z,\zeta) = \int_{{\mathbb R}}\varphi(y-z_2)e^{-2\pi i y\zeta_2}\left(V_\varphi\sigma(\cdot,y)\right)(z_1,\zeta_1)\, dy.
$$ From Lemma \ref{lem:sigma-mod-2-previo} there is a constant $C>0$ such that 
$$
\begin{array}{*2{>{\displaystyle}l}}
e^{N\omega(\zeta_1)}\left|\left(V_{\varphi\otimes\varphi}\sigma\right)(z,\zeta)\right|& \leq C\int_{|y-z_2|\leq 1}e^{k\omega(y)}|\varphi(y-z_2)|\, dy\\ & \\ & \leq Ce^{k\omega(1)}\int_{{\mathbb R}}|\varphi(y)|\, dy\cdot e^{k\omega(z_2)}.
\end{array} 
$$
\end{proof}

\begin{prop}\label{lem:sigma-mod-1} Let $\phi:{\mathbb R}^d\to {\mathbb R}^d$ be a smooth function such that for some $b\in [0,1),$ $|\phi(x)|=O(|x|^b)$ as $|x| \to \infty$  and let $\sigma(x,y) = e^{i\phi(x)y}.$ For every $N\in {\mathbb N}$ there is $k\in {\mathbb N}$ such that  $\sigma\in M^\infty_m({\mathbb R}^2),$ where $m(z,\zeta) = e^{N\omega(\zeta_2)-k\omega(|z_1|^b)}.$
\end{prop}
\begin{proof}
 In fact, the proof of Lemma \ref{lem:sigma} (2) works by taking $\varphi \in {\mathcal D}_\omega({\mathbb R}^d)$ and using Paley-Wiener's theorem \cite[Proposition 3.4]{BMT_90}. 
 \end{proof} 

 \begin{rem}{\rm 
 In case $b\in (0,1)$ and $\omega (t^b)=o(\omega(t))$ as $t\to \infty$, we get $\sigma\in M^\infty_{m}({\mathbb R}^{2d})$ where 
	$$m(z,\zeta) = e^{N\omega(\zeta_2)-c\omega(z_1)}$$ and $c>0$ is arbitrary, and if $b=0,$ then $\sigma\in M^\infty_{m}({\mathbb R}^{2d})$ where $m(z,\zeta)=e^{N\omega(\zeta_2)}. $
}
\end{rem} 

For every $s>0$ and $k > 0$ we denote
	$$
	v_s(z) = e^{s\left(\omega(z_1)+\omega(z_2)\right)},\ \ m_{s',k}(z) = v_{s'}(z)e^{k\omega(z_2)},\ \ z = (z_1, z_2).$$
    
\begin{thm}\label{th:main-onevariable} \begin{enumerate} \item Let $\phi '\in {\mathcal B}_{L_\infty, \omega}({\mathbb R})$ be given, and let $\sigma(x,y) = e^{i\phi(x)y}.$ We assume that for some $a\in[0,1)$ the following hold: $|\phi(x)|=O(|x|^a)$ and $\omega(x^a)=o(\omega(x))$ as $x\to \infty.$ Then for every $0<s<s'$ there exists $k>0$ such that $$\sigma(x,D):M^p_{m_{s',k}}({\mathbb R})\to M^p_{v_s} ({\mathbb R}) $$ for every $1\leq p \leq \infty.$

\item Given $\psi(x)=x+\phi(x),$  with  $\phi $ satisfying the hypothesis in (1). Then, for every $0<s<s'$ there exists $k>0,$ such that for each $1\leq p \leq \infty$ the composition operator $C_\psi$ acts continuously between $M^p_{m_{s',k}}({\mathbb R}) $ and $M^p_{v_s}({\mathbb R}).$
\end{enumerate}
\end{thm}
\begin{proof} 
(1)  We fix $N > 2s'$ and take $k$ as in Lemma \ref{lem:sigma-mod-2}. For brevity, we denote $m(z) = m_{s',k}(z).$ By interpolation (see \cite[Proposition 1.11]{TPT_25}) it suffices to prove the result for $p= 1$ and $p=\infty.$
 \par\medskip 
We first consider the case $p = 1.$ 
\par\medskip 
Let $\varphi\in {\mathcal S}_{\omega}({\mathbb R})$ a non-zero window and consider $\Phi_0 = R(\varphi, \varphi).$ For any $f,g\in {\mathcal S}_{\omega}({\mathbb R})$ we have
$$
\begin{array}{*2{>{\displaystyle}l}}
	\langle \sigma(x,D)f, g\rangle & = \int_{{\mathbb R}^{2}}\sigma(x,y)\overline{R(g,f)}(x,y)\, d(x,y) \\ & \\ & = \int_{{\mathbb R}^{4}}	\left(V_{\Phi_0}\sigma\right)(z,\zeta) \overline{V_{\Phi_0}(R(g,f))}(z,\zeta)\, d(z,\zeta).
\end{array}$$
Then $\left|\langle \sigma(x,D)f, g\rangle\right|$ is less than or equal to a constant times 
$$
\begin{array}{*1{>{\displaystyle}l}}
\int_{{\mathbb R}^{4}}	\left|\left(V_{\Phi_0}\sigma\right)(z,\zeta)\right|\cdot \left|V_\varphi g(z_1, z_2+\zeta_1)\right| \cdot\left|V_\varphi f(z_1 + \zeta_2, z_2)\right|\, d(z,\zeta)\\  \\ \leq \|g\|_{M^\infty_{v_{-s}}} \int_{{\mathbb R}^{4}}\left|\left(V_{\Phi_0}\sigma\right)(z,\zeta)\right| e^{-k\omega(z_2)} v_s(z_1, z_2+\zeta_1) \frac{\left|F(z_1+\zeta_2, z_2)\right|}{v_{s'}(z_1+\zeta_2, z_2)}\, d(z, \zeta)\\
\end{array}$$ where 
$$
F(a,b) = m(a,b) V_\varphi f (a,b).
$$
It turns out that
$$
\int_{{\mathbb R}^{2}}\left|F(z_1+\zeta_2, z_2)\right|\, dz = \int_{{\mathbb R}^{2}}\left|F(z_1, z_2)\right|\, dz = \|f\|_{M^1_{m}}$$ is independent of $\zeta_2.$ Moreover, having in mind that $v_s$ and $v_{s'}^{-1}$ are respectively $v_s$ and $v_{s'}$-moderate we have that
$$
\left|\left(V_{\Phi_0}\sigma\right)(z,\zeta)\right| e^{-k\omega(z_2)} v_{s'}^{-1}(z_1+\zeta_2, z_2)v_s(z_1, z_2 + \zeta_1)$$ is less than or equal to some constant times
$$
\left|\left(V_{\Phi_0}\sigma\right)(z,\zeta)\right| e^{-k\omega(z_2)}e^{(s-s')\omega(z_1)} v_{s'}(\zeta_2,0) v_s(0, \zeta_1).$$ Now  apply Propositions \ref{lem:sigma-mod-1} and \ref{lem:sigma-mod-2}  and the fact that $\omega(x^a)=o(\omega(x))$ as $x\to \infty$ to conclude
$$
\left|\left(V_{\Phi_0}\sigma\right)(z,\zeta)\right| e^{-k\omega(z_2)}e^{(s-s')\omega(z_1)}\lesssim \min\left(e^{-N\omega(\zeta_1)}, e^{-N\omega(\zeta_2)} \right).$$ Hence
$$
\left|\langle \sigma(x,D)f, g\rangle\right|	\lesssim \|f\|_{M^1_{m}} \|g\|_{M^\infty_{1/v_s}} \int_{{\mathbb R}^{2}} G(\zeta)\ d\zeta$$ where
$$
G(\zeta) = \min\left(e^{-N\omega(\zeta_1)}, e^{-N\omega(\zeta_2)} \right)v_{s'}(\zeta_2,0) v_s(0, \zeta_1).
$$ This means that $\sigma(x,D)f$ defines a continuous linear form on $M^0_{1/v_s},$ that is, $\sigma(x,D)f\in M^1_{v_s}$ and $\|\sigma(x,D)f\|_{M^1_{v_s}}\lesssim \|f\|_{M^1_m}$ for every $f\in {\mathcal S}_\omega({\mathbb R}).$ Thus $\sigma(x,D): {\mathcal S}_\omega({\mathbb R})\to  {\mathcal S}^\prime_\omega({\mathbb R})$ can be extended as a bounded operator 
$$
\sigma(x,D):M^1_{m_{s',k}}({\mathbb R}) \to M^1_{v_s}({\mathbb R}).
$$
\par\medskip 
Regarding the case $p=\infty, $ for $f,g\in {\mathcal S}_\omega({\mathbb R}^d)$ we have 
$$
\begin{array}{*1{>{\displaystyle}l}}
\left|\langle \sigma(x,D)f, g\rangle\right|	\lesssim \|f\|_{M^\infty_{m}}\times \\ \\ \int_{{\mathbb R}^{4}}\left|\left(V_{\Phi_0}\sigma\right)(z,\zeta)\right| e^{-k\omega(z_2)} v_{s'}^{-1}(z_1+\zeta_2, z_2) \left|V_\varphi g(z_1, z_2+\zeta_1)\right|\, d(z, \zeta)\\  \\  \leq \|f\|_{M^\infty_{m}}\times \\ \\ \int_{{\mathbb R}^{4}}\left|\left(V_{\Phi_0}\sigma\right)(z,\zeta)\right| e^{-k\omega(z_2)} \frac{v_s(z_1, z_2 + \zeta_1)}{v_{s'}(z_1+\zeta_2, z_2)} |F(z_1, z_2+\zeta_1)| \, d(z,\zeta)
\end{array}$$ where
$$
F(a,b) = v_s^{-1}(a,b) V_\varphi g (a,b).$$ It turns out that
$$
\int_{{\mathbb R}^{2}}\left|F(z_1, z_2+\zeta_1)\right|\, dz = \int_{{\mathbb R}^{2}}\left|F(z_1, z_2)\right|\, dz = \|g\|_{M^1_{1/v_s}}$$ is independent of $\zeta_1.$ As in the case $p=1,$ 
$$
\left|\left(V_{\Phi_0}\sigma\right)(z,\zeta)\right| e^{-k\omega(z_2)} v_{s'}^{-1}(z_1+\zeta_2, z_2)v_s(z_1, z_2 + \zeta_1)\leq G(\zeta)$$ where 
$$
\int_{{\mathbb R}^2} G(\zeta)\, d\zeta < \infty.$$ Consequently 
$$
\left|\langle \sigma(x,D)f, g\rangle\right|	\lesssim \|f\|_{M^\infty_{m}} \|g\|_{M^1_{1/v_s}}
$$ for every $f,g\in {\mathcal S}_\omega({\mathbb R}).$	Since we already know that $\sigma(x,D)({\mathcal S}_\omega({\mathbb R}))$ is contained in $M^1_{v_s}({\mathbb R})\subset  M^0_{v_s}({\mathbb R}),$ this implies that $\sigma(x,D)$ extends to a bounded operator $$\sigma(x,D):M^0_m({\mathbb R}) \to M^0_{v_s}({\mathbb R}).$$ Taking the bi-transpose map, $$\sigma(x,D):M^\infty_m({\mathbb R}) \to M^\infty_{v_s}({\mathbb R})$$ is continuous.

(2) follows directly from (1). 
\end{proof}

As a consequence of the previous result, we obtain an improvement of \cite[4.3]{jmaa}:

\begin{cor} Let $\psi(x) = Ax + \phi(x)$  for some $A\neq 0,$ $\phi '\in {\mathcal B}_{L_\infty, \omega}({\mathbb R})$ and $|\phi(x)|=O(|x|^a)$  for some $a \in [0,1).$ We assume that  $\omega(x^a)=o(\omega(x))$ as $x\to \infty$. Then the composition operator $$C_\psi: {\mathcal S}_{\omega}({\mathbb R})\to {\mathcal S}_{\omega}({\mathbb R})$$ is well defined and continuous.
 \end{cor}

\begin{rem}{\rm  \begin{enumerate}
 \item Let $\omega $ be  a weight  function, then, by property ($\gamma$),  $f(x)=e^{-A'\omega(x)} \in L^2_{e^{A\omega}}({\mathbb R}) ,$ ($A' \, A>0$), if and only if $A'>A.$ If the weight satisfies the 
    (BMM)-condition,  (see \cite[p. 435]{bmm}):
      $$ \exists H>1 \, : \, 2\omega (t)\leq \omega(Ht)+H \mbox{ for every }\, t\geq 0, $$   
      given $A>0$ we cannot find $B>0$ such that $L^2_{e^{B\omega}}({\mathbb R})$ contains all dilations of $L^2_{e^{A\omega}}({\mathbb R}).$  In fact,  since for every $n\in {\mathbb N},$ $2^n\omega(t)\leq \omega(H^nt)+(2^n-1)H$, $f(\frac{t}{H^n})\in L^2_{e^{B\omega}}({\mathbb R}) $ if and only if $2^nB<A'.$  We observe that $M^2_{e^{A\omega}\otimes {\bm 1}}=L^2_{e^{A\omega}}({\mathbb R}).$
    On the other hand, the dilations also do not leave the space $M^2_{{\bm 1}\otimes e^{A\omega}}({\mathbb R})={\mathcal F}(L^2_{e^{A\omega}}({\mathbb R}))$ invariant when $A > 0$ since ${\mathcal F}(g_\lambda)=\lambda^{-1}({\mathcal F}(g))_{1/\lambda}.$ The Gevrey weights, $\omega(t)=t^{1/s},$ $s>1,$ satisfy property (BMM).
    \item If a weight $\omega$ satisfies condition (BMM), then $\omega(x^a)=o(\omega(x))$ as $x\to \infty$ for each $a\in [0,1)$ (see the proof of \cite[3.6]{jmaa}).
\end{enumerate}
}    
\end{rem}

\subsection{Several variables}
We need a variant of Lemma \ref{lem:sigma-derivadas}.

\begin{lem}\label{lem:sigma-derivadas-improved}
Let $\phi = (\phi_1,\ldots,\phi_d):{\mathbb R}\to {\mathbb R}^d$ be given such that $\phi'_j\in {\mathcal B}_{L_\infty,\omega}$ for every $1\leq j\leq d.$ We put $\sigma(x,y) = e^{i\phi(x)y},\ x\in {\mathbb R},\ y\in {\mathbb R}^d.$ Then for every $\ell\in{\mathbb N}$ there exists $m_\ell > 0$ such that
$$
\left|\left(\partial_x^n\sigma\right)(x,y)\right| \leq (4d)^n e^{d m_\ell\omega(y)} e^{\ell\varphi^\ast_\omega(\frac{n}{\ell})}\ \ \forall n\in {\mathbb N}.$$
\end{lem} 
\begin{proof}
We  denote $\sigma_j(x,y) = e^{i\phi_j(x)y_j},\ 1\leq j\leq d.$ Then 
$$
\partial_x^n\sigma(x,y) = \sum_{n_1+\ldots+n_d=n}\binom{n}{n_1,\ldots,n_d}\prod_{j=1}^d\partial_x^{n_j}\sigma_j.$$ From Lemma \ref{lem:sigma-derivadas} we get 
$$
\begin{array}{*2{>{\displaystyle}l}}
\left|\left(\partial_x^n\sigma\right)(x,y)\right| & \leq \sum_{n_1+\ldots+n_d=n}\binom{n}{n_1,\ldots,n_d}\prod_{j=1}^d 4^{n_j}e^{m_\ell \omega(y_j)}e^{\ell\varphi^\ast_\omega(\frac{n_j}{\ell})}\\ & \\ & \leq 4^ne^{d m_\ell\omega(y)} e^{\ell\varphi^\ast_\omega(\frac{n}{\ell})}\sum_{n_1+\ldots+n_d=n}\binom{n}{n_1,\ldots,n_d}.
\end{array}
$$
\end{proof}
 Proceeding as in the proof of Lemma \ref{lem:sigma-mod-2-previo} but applying Lemma \ref{lem:sigma-derivadas-improved} instead of Lemma \ref{lem:sigma-derivadas} we obtain the following. Observe that the constant $k$ only depends on $N,$ the weight $\omega$ and the norms in ${\mathcal B}_{L_\infty,\omega}$ of $\phi'_j,\ 1\leq j\leq d.$

 \begin{lem}\label{lem:sigma-mod-2-previo-improved}
 Under the hypothesis of Lemma \ref{lem:sigma-derivadas-improved}, for every $N\in {\mathbb N}$ there is $k\in {\mathbb N}$ such that  
$$
\left\{e^{-k\omega(y)}\sigma(\cdot,y):\ y\in {\mathbb R}^d\right\}
$$ is a bounded set in $M^\infty_m({\mathbb R}),$ where $m(z_1,\zeta_1) = e^{N\omega(\zeta_1)}.$ 
 \end{lem}

 \begin{lem}\label{lem:sigma-mod-1-several} Let $\phi = (\phi_1, \ldots, \phi_d)$ be given such that $\frac{\partial \phi_i}{\partial x_k}\in {\mathcal B}_{L_\infty,\omega}({\mathbb R}^d)$ for every $1\leq i,j\leq d$ and $$\sigma(x,y) = e^{i\phi(x) y},\ \ x,y\in {\mathbb R}^d.$$ For every $N\in {\mathbb N}$ and $k = 1, \ldots, d$ there is $p > 0$ such that $\sigma\in M^\infty_m({\mathbb R}^{2d})$ where
	$$
	m(z,\zeta) = e^{N\omega(\zeta_k)-p\omega(z_{d+1},\ldots,z_{2d})},\ \ z,\zeta\in {\mathbb R}^{2d}.$$
\end{lem}
\begin{proof} For every $x\in {\mathbb R}^d$ and $1\leq k\leq d$ we denote
	$$\widetilde{x_k} = \left(x_1, \ldots, x_{k-1}, 0, x_{k+1},\ldots, x_d\right)$$ and 
    $$
    \sigma_{\widetilde{x_k}}(t,y) = \sigma(\widetilde{x_k}+te_k,y),\ t\in {\mathbb R}, y\in {\mathbb R}^d.$$ Now we consider  $\psi = \overbrace{\varphi\otimes\ldots\otimes\varphi}^{(2d)}$ and observe that $V_\psi\sigma(z,\zeta)$ is given by
     $$
    \begin{array}{*1{>{\displaystyle}l}}
	\int_{{\mathbb R}^{2d-1}}\prod_{j\neq k} \varphi(x_j-z_j)\prod_{j=1}^d \varphi(y_j-z_{j+d}) \ \times \\ \\ \times \exp\left(-2\pi i \left(\displaystyle\sum_{j\neq k}^d \zeta_j x_j + \sum_{j=1}^d \zeta_{j+d}y_j\right)\right)\left(V_\varphi \sigma_{\widetilde{x_k}}(\cdot,y)\right)(z_k,\zeta_k)\, d(\widehat{x_k},y),\end{array} 
$$ where $$
	\widehat{x_k} = \left(x_1, \ldots, x_{k-1}, x_{k+1},\ldots, x_d\right).$$ Since the family of functions
$$
t\mapsto \frac{\partial \phi_j}{\partial x_k}(\widetilde{x_k}+t e_k),\ \ x\in {\mathbb R}^d,\ 1\leq j\leq d, \ 1\leq k\leq d,$$ is bounded in ${\mathcal B}_{L_\infty,\omega}({\mathbb R}),$ we can apply Lemma \ref{lem:sigma-mod-2-previo-improved} to obtain
 $$
\begin{array}{*2{>{\displaystyle}l}}
\left|V_\psi\sigma(z,\zeta)\right| & \lesssim \|\varphi\|_1^{d-1}e^{-N\omega(\zeta_k)}\int_{{\mathbb R}^{d}}\prod_{j=1}^d |\varphi(y_j-z_{j+d})| e^{k\omega(y)}\, dy \\ & \\ & \lesssim e^{-N\omega(\zeta_k)} e^{k\omega(z_{d+1},\ldots,z_{2d})}\left(\int_{{\mathbb R}} \varphi(t)e^{k\omega(t)}\, dt\right)^d.
\end{array}$$
\end{proof}

For $z\in {\mathbb R}^{2d}$ we will denote $\overline{z_1}:=(z_1,\ldots, z_d)$ and $\overline{z_2}:=(z_{d+1},\ldots, z_{2d}).$ 

From Lemma \ref{lem:sigma-mod-1-several} and (the proof of) Lemma \ref{lem:sigma-mod-1} we get the following.

\begin{prop} Let $\phi = (\phi_1, \ldots, \phi_d)$ be given such that $\frac{\partial \phi_i}{\partial x_k}\in {\mathcal B}_{L_\infty,\omega}({\mathbb R}^d)$ for every $1\leq i,j\leq d$ and $$\sigma(x,y) = e^{i\phi(x) y},\ \ x,y\in {\mathbb R}^d.$$ For every $N\in {\mathbb N}$ there exists $p>0$ such that $\sigma\in M^\infty_{m_1}({\mathbb R}^d)\cap M^\infty_{m_2}({\mathbb R}^d)$ where
$$
m_1(z,\zeta) = e^{N\omega(\overline{\zeta_1})-p\omega(\overline{z_2})},\ \ m_2(z,\zeta) = e^{N\omega(\overline{\zeta_2})-p\omega(\overline{z_1})}.$$
\end{prop}

The previous result provides the extension of Theorem \ref{th:main-onevariable} to the multivariable setting, whose proof is omitted.

\begin{thm} Let $\phi = (\phi_1, \ldots, \phi_d)$ be given such that $\frac{\partial \phi_i}{\partial x_j}\in {\mathcal B}_{L_\infty,\omega}({\mathbb R}^d)$ for every $1\leq i,j\leq d$ and $|\phi(x)| = O(|x|^b)$ for some $0\leq b < 1.$ We put $$\sigma(x,y) = e^{i\phi(x) y},\ \ x,y\in {\mathbb R}^d.$$ For every $s, k > 0$ we denote
	$$
v_s(z) = e^{s(\omega(\overline{z_1})+\omega(\overline{z_2}))},\ \ m_{s,k}(z) = v_s(z)e^{k\omega(\overline{z_2})}.$$ Then for every $s\in {\mathbb N}$ (large enough) there exists $k>0$ such that
$$
\sigma(x,D):M^p_{m_{s,k}}({\mathbb{R}}^d) \to M^p_{v_s}({\mathbb{R}}^d).$$

If $\psi(x)=x+\phi(x)$ with $\phi$ as above, then the composition operator $C_\psi$ acts continuously between $M^p_{m_{s,k}}({\mathbb{R}}^d)$ and $M^p_{v_s}({\mathbb{R}}^d).$
\end{thm}

The next corollary follows immediately from  the preceding result, together with the invariance of ${\mathcal S}_\omega({\mathbb R}^d) $ under  linear changes of coordinates  and identity (\ref{Gelfand-Shilov}).

\begin{cor} Let us assume  that $\psi(x)=Ax+\phi(x),$ $A$ being an invertible $d\times d$ matrix,  $\partial_{x_j}  \phi\in {\mathcal B}_{L_\infty, \omega}(\mathbb{R}^d, \mathbb{R}^d)$ for each $1\leq j \leq d$ and $|\phi(x)| = O(|x|^b)$ as $|x|\to \infty$ for some  $b\in [0,1).$  Then  $$C_\psi:{\mathcal S}_\omega({\mathbb R}^d)\to {\mathcal S}_\omega({\mathbb R}^d) $$ is continuous.
\end{cor}

\end{document}